\documentclass[11pt,reqno]{amsart}
\usepackage{a4wide}
\usepackage{amssymb}
\usepackage{amsthm}
\usepackage{amsmath,todonotes}
\usepackage{amscd}
\usepackage{cancel}
\usepackage{verbatim}
\usepackage{color}
\usepackage[all]{xy}
\usepackage[mathscr]{eucal}
\usepackage{mathrsfs}
\usepackage{fullpage}
\usepackage{enumitem}
\usepackage{bbm}
\usepackage{qtree}
\usepackage{bm}
\usepackage{bigints}
\usepackage[headheight=110pt,top=0.7in, bottom=0.7in, left=0.6in, right=0.6in]{geometry}
\allowdisplaybreaks

\numberwithin{equation}{section}

\newtheorem{theorem}{Theorem}

\newtheorem{lemma}[theorem]{Lemma}

\newtheorem{corollary}[theorem]{Corollary}

\theoremstyle{definition}

\numberwithin{theorem}{section} 
\numberwithin{equation}{section}
\numberwithin{table}{section}

\newcommand{\Z}{\mathbb{Z}}
\newcommand{\N}{\mathbb{N}}

\newcommand{\re}{\textnormal{Re}}
\newcommand{\im}{\textnormal{Im}}

\renewcommand{\(}{\left(}

\begin{document}
\title[Higher power of odd zeta values]{Transformation formulas for the higher power of odd zeta values and generalized Eisenstein series}
\author{Soumyarup Banerjee}
\address{ Discipline of Mathematics, Indian Institute of Technology Gandhinagar, Palaj, Gandhinagar, Gujarat - 382355, India.}
\email{soumyarup.b@iitgn.ac.in}
\author{Vijay Sahani}
\address{Discipline of Mathematics, Indian Institute of Technology Gandhinagar, Palaj, Gandhinagar, Gujarat - 382355, India.}
\email{vijays@iitgn.ac.in}



\medskip
\begin{abstract}
In this article, we obtain a transformation formula for the higher power of odd zeta values, which generalizes Ramanujan's formula for odd zeta values. We have also investigated many important applications, which in turn provide generalizations of the transformation formula of the Eisenstein series, Dedekind eta function etc. 
\end{abstract}

\maketitle
\section{Introduction}
The Riemann zeta function $\zeta(s)$ is a central object in analytic number theory to study the distribution of primes and has applications in physics, probability theory, and applied statistics. Over the years, the special values of $\zeta(s)$ and its arithmetic behaviour have drawn the attention of many mathematicians. The story begins with a remarkable discovery of Euler in 1734 about determining the special values of $\zeta(s)$ at positive even integral arguments. Euler's formula states that for any natural number $m$,
\begin{equation}
\zeta(2m) = (-1)^{m+1} \frac{(2\pi)^{2m} B_{2m}}{2(2m)!}. \label{Euler's formula}
\end{equation}
where $B_m$ denotes the $m$-th Bernoulli number. The above formula readily implies that all even zeta values $\zeta(2m)$ with $m \in \N$, are transcendental due to the well-known fact that $\pi$ is transcendental and $B_m$ is rational.

The arithmetic nature of zeta values at odd integral arguments is still far from being known. The irrationality of $\zeta(3)$ was established by Ap{\'e}ry \cite{Apery1}, \cite{Apery2}. Recently, Zudilin \cite{Zudilin} has shown that at least one of the four numbers $\zeta(5)$, $\zeta(7)$, $\zeta(9)$, $\zeta(11)$ is irrational. 

In this context, Ramanujan \cite[p.~173, Ch.~14, Entry~21(i)]{Ramanujan} obtained an elegant identity for odd zeta values, which precisely states that for any non-zero integer $m$ and $\alpha, \beta >0$ with $\alpha \beta = \pi^2$, 
\begin{align}
\alpha^{-m}\left\lbrace \frac{1}{2}\zeta(2m+1) + \sum_{n=1}^\infty\frac{n^{-2m-1}}{e^{2n\alpha} - 1} \right\rbrace
= (-\beta)^{-m}\left\lbrace \frac{1}{2}\zeta(2m+1) + \sum_{n=1}^\infty\frac{n^{-2m-1}}{e^{2n\beta} - 1} \right\rbrace \nonumber\\
-2^{2m} \sum_{j=0}^{m+1} \frac{(-1)^j B_{2j} B_{2m+2-2j}}{(2j)!(2m+2-2j)!}\alpha^{m+1-j}\beta^j.\label{Ramanujan formula}
\end{align}

As an immediate application of the above identity, one can consider $\alpha = \beta = \pi$ and $m$ to be odd in \eqref{Ramanujan formula} to obtain Lerch's identity \cite{Lerch}, given by,
\begin{equation}
\zeta(2m+1) + 2\sum_{n=1}^\infty \frac{1}{n^{2m+1} (e^{2\pi n}-1)} = 2^{2m}\pi^{2m+1} \sum_{j=0}^{m+1} \frac{(-1)^{j+1} B_{2j} B_{2m+2-2j}}{(2j)!(2m+2-2j)!}.\label{Lerch}
\end{equation}
Thus it follows from the above identity that for $m$ odd, at least one of $\zeta(2m+1)$ and $\sum_{n=1}^\infty \frac{1}{n^{2m+1} (e^{2\pi n}-1)}$ is transcendental.

In \eqref{Ramanujan formula}, Ramanujan expressed $\zeta(2m+1)$ in terms of the Lambert series
\begin{equation}\label{Lambert series in Ramanujan formula}
\sum_{n=1}^\infty \frac{n^a}{e^{ny}-1} = \sum_{n=1}^\infty \sigma_a(n) e^{-ny}
\end{equation} 
with $a = -2m-1$ and $\re(y)>0$. Here $\sigma_a(n)$ denotes the general divisor function defined by $\sigma_a(n) := \sum_{d\mid n} d^a$. For $a = 2m-1$ with $m\in \N$ and $y = -2\pi iz$ with $z \in \mathbb{H}$, the upper half-plane, either of the above series essentially represents the Eisenstein series $E_{2m}$ of weight $2m$ on the full modular group $SL_2 (\Z)$. The transformation formula satisfied by $E_{2m}$ with $m>1$ over $SL_2(\Z)$ under the transformation $z \mapsto -\frac{1}{z}$ are namely, for $\alpha, \beta>0$ with $\alpha\beta=\pi^2$,
\begin{equation}\label{Eisenstein series}
\alpha^m \sum_{n=1}^\infty \frac{n^{2m-1}}{e^{2\alpha n}-1}- (-\beta)^m \sum_{n=1}^\infty \frac{n^{2m-1}}{e^{2\beta n}-1} = (\alpha^m- (-\beta)^m)\frac{B_{2m}}{4m}.
\end{equation}
The series $\sum_{n=1}^\infty \sigma_1(n)e^{2\pi inz}$ is essentially the quasi-modular form $E_2(z)$, that is, the Eisenstein series of weight 2 over $SL_2(\Z)$. The transformation formula of $E_2(z)$ under the transformation $z \mapsto -\frac{1}{z}$ are namely, for $\alpha, \beta>0$ with $\alpha\beta=\pi^2$,
\begin{equation}\label{Quasi modular form}
\alpha \sum_{n=1}^\infty \frac{n}{e^{2\alpha n}-1} + \beta \sum_{n=1}^\infty \frac{n}{e^{2\beta n}-1} = \frac{\alpha+\beta}{24}-\frac{1}{4}.
\end{equation}
When $a = -2m+1$, the above series represents the Eichler integral corresponding to the weight $2m$ Eisenstein series (cf. \cite[Section 5]{Berdnt}). Thus, Ramanujan's identity \eqref{Ramanujan formula} encapsulates fundamental modular properties of even integral weight Eisenstein series over the full modular group and their Eichler integrals. The identity \eqref{Ramanujan formula} also has an application in theoretical computer science \cite{Kirschenhofer}, in particular, in the analysis of special data structures and algorithms.

Malurkar \cite{Malurkar} first independently obtained the proof of \eqref{Ramanujan formula}. Later, the formula was rediscovered by Grosswald \cite{Grosswald1} \cite{Grosswald2}, where he studied it more generally. Both Euler's formula \eqref{Euler's formula} and Ramanujan's formula \eqref{Ramanujan formula} follow as a special case of Berdnt's general transformation formula \cite[Theorem 2.2]{Berdnt0}, which in turn shows that Euler's and Ramanujan's formulas are natural companions of each other. Recently, O'Sullivan \cite[Theorem 1.3]{Sullivan} studied non-holomorphic analogues of the formulas of Ramanujan, Grosswald, and Berndt containing Eichler integrals of holomorphic Eisenstein series.

Let $d(n) := \sum_{d\mid n}1$ be the divisor function. The Dirichet series associated to the divisor function $d(n)$ is precisely $\sum_{n=1}^\infty d(n)n^{-s} = \zeta^2(s)$, which was studied by numerous mathematicians in different directions (cf. \cite{Ferrar}, \cite{Oberhettinger}, \cite{Voronoi}) from the point of view of analytic number theory. 

The formula for $\zeta^2(2m)$ can be obtained by squaring both sides of \eqref{Euler's formula}, but the resulting formula for $\zeta^2(2m+1)$ after squaring both  sides of \eqref{Ramanujan formula} is extremely intricate and could not be simplified further. Recently, Dixit and Gupta \cite[Theorem 2,1]{Dixit} established a transformation formula for $\zeta^2(2m+1)$, $m \in \Z \setminus \{0\}$, which can be considered as an analogue of Ramanujan's identity \eqref{Ramanujan formula}.

Koshliakov \cite{Koshliakov} studied a function, namely,
\begin{equation}\label{Omega}
\Omega(x) := 2\sum_{j=1}^\infty d(j) \left(K_0 \left(4\pi \epsilon \sqrt{jx}\right) + K_0 \left(4\pi \overline{\epsilon} \sqrt{jx}\right)\right),
\end{equation}
where $\epsilon = \exp\left(\frac{i \pi}{4}\right)$ and $K_z(x)$ denotes the modified Bessel function of the second kind of order $z$ (cf. \cite[p.~78]{Watson}), which is defined later in \S \ref{Preliminaries}. In the same article, Koshliakov established two beautiful identities \cite[Equations (27), (29)]{Koshliakov}, which derives an infinite series involving $\Omega(x)$ as follows :
\begin{equation*}
\sum_{n=1}^\infty n^{4m+1} d(n) \Omega(n) = \frac{B^2_{4m+2}}{(4m+2)^2}\left[\log(2\pi) - \sum_{k=1}^{4m+1}\frac{1}{k}- \frac{\zeta'(4m+2)}{\zeta(4m+2)} \right]
\end{equation*}
and
\begin{equation*}
\sum_{n=1}^\infty n d(n) \Omega(n) = \frac{1}{144}\left[\log(2\pi) - 1 - \frac{6}{\pi^2}\zeta'(2) \right] - \frac{1}{32\pi}.
\end{equation*} 

In \cite{Dixit}, Dixit and Gupta generalized the function $\Omega(x)$ in \eqref{Omega} by introducing a new parameter $\rho$ as
\begin{equation}\label{Generalized Koshliakov's function}
\Omega_\rho(x) := 2\sum_{j=1}^\infty d(j) \left(K_0 \left(4\rho \epsilon \sqrt{jx}\right) + K_0 \left(4\rho \overline{\epsilon} \sqrt{jx}\right)\right)
\end{equation}
such that $\Omega_\pi(x) = \Omega(x)$ and obtained the transformation formula for $\zeta^2(2m+1)$, which precisely states that for any non-zero integer $m$ and $\alpha, \beta >0$ with $\alpha \beta = \pi^2$, 
\begin{multline}
\alpha^{-2m} \left[\zeta^2(2m+1)\left(\gamma+\log\left(\frac{\alpha}{\pi} \right)- \frac{\zeta'(2m+1)}{\zeta(2m+1)}\right) + \sum_{n=1}^\infty \frac{d(n)\Omega_\alpha(n)}{n^{2m+1}} \right]\\
= (-1)^{-m}\beta^{-2m} \left[\zeta^2(2m+1)\left(\gamma+\log\left(\frac{\beta}{\pi} \right)- \frac{\zeta'(2m+1)}{\zeta(2m+1)}\right) + \sum_{n=1}^\infty \frac{d(n)\Omega_\beta(n)}{n^{2m+1}} \right]\\
- 2^{4m} \pi \sum_{j=0}^{m+1} \frac{(-1)^j B_{2j}^2 B_{2m+2-2j}^2}{(2j)!^2 (2m+2-2j)!^2} \alpha^{2j} \beta^{2m+2-2j}. \label{Dixit formula}
\end{multline}

Let $d_k(n)$ denotes the number of ways $n$ can be written as $k$ given factors, which is sometimes known as Piltz divisor function. We define an infinite series $\Psi_{\rho, k}(x)$ involving the Meijer G-function (cf. \S \ref{sec:specialfunctions}). This series includes the Lambert series \eqref{Lambert series in Ramanujan formula} that appeared in Ramanujan's formula \eqref{Ramanujan formula}, and the generalized Koshliakov's function \eqref{Generalized Koshliakov's function} appeared in the identity \eqref{Dixit formula} of Dixit and Gupta as a special case. For $\rho>0$, let $\Psi_{\rho, k}(x)$ be defined as 
\begin{equation}\label{Psiomegakx}
\Psi_{\rho, k}(x) :=\frac{\pi^{k/2-1}}{2^{k-1}} \sum_{j=1}^{\infty}d_{k}(j) 
G_{0,\ 2k}^{k+1,\ 0}\bigg(
\begin{matrix}
\text{---}\\
(0)_k,\frac{1}{2}:\left(\frac{1}{2}\right)_{k-1}
\end{matrix} \ \bigg| \
\frac{\rho^2 j^2 x^2}{2^{2k}}
\bigg),
\end{equation}
where $k$ is any natural number. In this article, we have generalized Ramanujan's identity and the identity of Dixit and Gupta by studying  the transformation formula for $\zeta^k(2m+1)$ with $m\in \Z\setminus \{0\}$, where $k$ is any natural number.
\begin{theorem}\label{Main result}
For any non-zero integer $m$ and for any $\alpha, \beta >0$ satisfying $\alpha\beta=\pi^{2}$, we have
\begin{align}
&(\alpha^k)^{-m}\left[\sum_{n=1}^{\infty}\frac{d_k(n)}{n^{2m+1}}\Psi_{(2\alpha)^k, k}(n)-\frac{1}{(k-1)!}\frac{d^{k-1}}{ds^{k-1}}\left( \zeta^k(2m+1+s)\zeta^k(s)\Gamma^k(s+1)\cos^{k-1}\left(\frac{\pi s}{2}\right)(2\alpha)^{-ks}\right)\bigg\rvert_{s=0}\right]\nonumber\\
&=(-\beta^k)^{-m}\left[
\sum_{n=1}^{\infty}\frac{d_k(n)}{n^{2m+1}}\Psi_{(2\beta)^k, k}(n)-\frac{1} {(k-1)!}\frac{d^{k-1}}{ds^{k-1}}\left(\zeta^k(2m+1+s)\zeta^{k}(s)\Gamma^{k}(s+1)\cos^{k-1}\left(\frac{\pi s}{2}\right)(2\beta )^{-ks}\right)\bigg\rvert_{s=0}\right]\nonumber\\
&+(-1)^{km+k+m}\left(\frac{\pi}{2}\right)^{k-1} 2^{2km}\sum_{j=0}^{m+1}(-1)^{j}\frac{B^k_{2m-2j+2}B^k_{2j}}{(2m-2j+2)!^k(2j)!^k} \alpha^{k(m+1-j)}\beta^{kj}, \label{General Ramanujan identity}
\end{align}
where $k$ is any natural number.
\end{theorem}
As an immediate application of the above theorem, we obtain the following two results.
\begin{corollary}\label{Ramanujan theorem}
Ramanujan's identity \eqref{Ramanujan formula} for odd zeta values is valid.
\end{corollary}
\begin{corollary}\label{Dixit theorem}
The identity \eqref{Dixit formula} of Dixit and Gupta is valid.
\end{corollary}
In Theorem \ref{Main result}, the identity \eqref{General Ramanujan identity} can be analytically continued to any complex $\alpha, \beta$ with $\re(\alpha)>0$ and $\re(\beta)>0$. Thus, letting $\alpha=-\pi i z$ with any complex number $z$ in upper-half plane and $m=-m$ with $m>0$, the series in \eqref{General Ramanujan identity} namely,
\begin{equation*}
\sum_{n=1}^{\infty}d_k(n)n^{2m-1}\Psi_{(2\alpha)^k, k}(n)
\end{equation*} 
represents the generalization of usual Eisenstein series of weight $2m$ over $SL_2(\Z)$. The following corollary provides the generalization of the transformation formula \eqref{Eisenstein series} of $E_{2m}(z)$ for $m>1$ over $SL_2(\Z)$, which can be obtained directly from \eqref{General Ramanujan identity}, substituting $m$ by $-m$ .
\begin{corollary}
Let $m$ be any natural number with $m>1$. Then for any $\alpha, \beta >0$ satisfying $\alpha\beta=\pi^{2}$, we have
\begin{align*}
&(\alpha^k)^{m}\sum_{n=1}^{\infty}d_k(n)n^{2m-1}\Psi_{(2\alpha)^k, k}(n) - (-\beta^k)^{m}\sum_{n=1}^{\infty}d_k(n)n^{2m-1}\Psi_{(2\beta)^k, k}(n) \\
&= \frac{1}{(k-1)!}\frac{d^{k-1}}{ds^{k-1}}\left\{ \zeta^k(1-2m+s)\zeta^k(s)\Gamma^k(s+1)\cos^{k-1}\left(\frac{\pi s}{2}\right)2^{-ks}\left(\alpha^{-k(s-m)} - (-1)^m\beta^{-k(s-m)}\right)\right\}\bigg\rvert_{s=0}
\end{align*}
\end{corollary}
The next result generalizes the transformation formula \eqref{Quasi modular form} of Quasi modular form $E_2(z)$ over $SL_2(\Z)$, which can be obtained by inserting $m=-1$ in \eqref{General Ramanujan identity}.
\begin{corollary} 
Let $\alpha, \beta$ be any real number with $\alpha, \beta >0$, which satisfies $\alpha\beta=\pi^{2}$. Then the following identity
\begin{align*}
&\alpha^k\sum_{n=1}^{\infty}nd_k(n)\Psi_{(2\alpha)^k, k}(n) +\beta^k\sum_{n=1}^{\infty}nd_k(n)\Psi_{(2\beta)^k, k}(n) \\
&= \frac{1}{(k-1)!}\frac{d^{k-1}}{ds^{k-1}}\left\{ \zeta^k(s-1)\zeta^k(s)\Gamma^k(s+1)\cos^{k-1}\left(\frac{\pi s}{2}\right)2^{-ks}\left(\alpha^{-k(s-1)} + \beta^{-k(s-1)}\right)\right\}\bigg\rvert_{s=0} - \left(\frac{\pi}{2} \right)^{k-1}2^{-2k}
\end{align*}
holds.
\end{corollary}
For any complex number $z$ with $\im(z) > 0$, The Dedekind eta function $\eta(z)$ can be defined as
\begin{equation*}
\eta(z) := e^{\frac{\pi i z}{12}} \prod\limits_{n=1}^\infty \left( 1 - e^{2\pi i nz}\right).
\end{equation*}
It is a modular form of weight $\frac{1}{2}$ over the full modular group and has a Fourier series expansion $\eta(z) = \sum_{n=1}^{\infty} \sigma_{-1}(n)e^{2\pi i nz}$. Ramanujan's formula \eqref{Ramanujan formula} also provides a transformation formula of $\eta(z)$ under the transformation $z\mapsto \frac{1}{z}$ as a special case which is given by
\begin{equation*}
\sum_{n=1}^{\infty} \sigma_{-1}(n)e^{-2n\alpha}-\sum_{n=1}^{\infty} \sigma_{-1}(n)e^{-2n\beta}=\frac{\beta-\alpha}{12}+\frac{1}{4}\log\left(\frac{\alpha}{\beta}\right),
\end{equation*}
where $\alpha, \beta>0$ with $\alpha\beta=\pi^{2}$. We next generalize the above transformation formula. 
\begin{theorem}\label{Generalization of Dedekind eta}
For $\alpha, \beta >0$ satisfying $\alpha \beta=\pi^{2}$,
\begin{align*}
&\sum_{n=1}^{\infty} \frac{d_k(n)}{n}\Psi_{(2\alpha)^k,k}(n)-\sum_{n=1}^{\infty} \frac{d_k(n)}{n}\Psi_{(2\beta)^k,k}(n)\\
&=\frac{(-1)^k2^{k}}{(2k-1)!}\frac{d^{2k-1}}{ds^{2k-1}}\left( \Gamma^k(1+s)\Gamma^k(1-s)\zeta^k(s)\zeta^k(-s)\cos^{2k-1}\left(\frac{\pi s}{2}\right)\left(\frac{\alpha}{\beta}\right)^{-\frac{ks}{2}}\right)\bigg\rvert_{s=0}
\hspace{-3pt}+2\frac{(-1)^{k+1}\pi^{k-1}}{24^k}(\beta^k-\alpha^k)
\end{align*}
\end{theorem}
Clearly, one can obtain the transformation formula of $\eta(z)$ by putting $k=1$ in the above identity. For $k=2$, the above theorem reduces to \cite[Theorem 2.9]{Dixit}. 

The next result can be considered as a generalization of Lerch's identity \eqref{Lerch}, which follows directly from Theorem \ref{Main result} by inserting $\alpha = \beta = \pi^k$ in \eqref{General Ramanujan identity} and considering $m$ to be odd.
\begin{theorem}
Let $m$ be any non-zero odd integer. Then for any natural number $k$,
\begin{align*}
\sum_{n=1}^{\infty}\frac{d_k(n)}{n^{2m+1}}\Psi_{(2\pi)^k, k}(n)&=\frac{1}{(k-1)!}\Bigg[\frac{d^{k-1}}{ds^{k-1}}\left( \zeta^k(2m+1+s)\zeta^k(s)\Gamma^k(s+1)\cos^{k-1}\left(\frac{\pi s}{2}\right)(2\pi)^{-ks}\right)\bigg\rvert_{s=0}\\
&+2^{2km-k}\pi^{2km+2k-1}\sum_{n=0}^{m+1}(-1)^{n+1}\frac{B^k_{2m-2n+2}B^k_{2n}}{((2m-2n+2)!(2n)!)^k}.
\end{align*}
In particular, at $k=1$, the above identity reduces to Lerch's identity \eqref{Lerch}.
\end{theorem}
The case $k=2$ of the above identity was obtained in \cite[Equation (2.4)]{Dixit}.

The paper is organized as follows. In \S \ref{Preliminaries}, we collect basic tools, which we have applied throughout. \S \ref{Generalization of Ramanujan's identity and it's special cases} provides the proof of the transformation formula of $\zeta^k(2m+1)$, and its special cases. Finally, in \S \ref{Transformation of the generalized Dedekind eta function}, we generalized the transformation formula of $\eta(z)$.

\section{Preliminaries}\label{Preliminaries}
In this section, we collect some basic tools of analytic number theory and complex analysis, which will be applied throughout. 
\subsection{Gamma function}
The Gamma function plays a significant role in this paper. For $\re(z) > 0$, it can be defined via the convergent improper integral 
\begin{equation*}
\Gamma(z) = \int_0^\infty e^{-t} t^{z-1} {\rm d}t.
\end{equation*}

The function $\Gamma(z)$ is absolutely convergent for $\mathfrak{R}(z) > 0$. It can be analytically continued to the whole complex plane except for simple poles at every non-positive integers. It also satisfies the functional equation \cite[Appendix A]{Ayoub}, namely,
\begin{equation*}
\Gamma(z+1) = z\Gamma(z).
\end{equation*}

Two important properties of Gamma function which will be applied frequently is the following :
\begin{equation}\label{Reflection formula}
\Gamma(z)\Gamma(1-z) = \frac{\pi}{\sin \pi z},
\end{equation}
where $z \notin \mathbb{Z}$ and
\begin{equation}\label{Duplication formula}
\Gamma(z)\Gamma \left(z+\frac{1}{2}\right) = 2^{1-2z} \sqrt{\pi} \Gamma(2z)
\end{equation}
for any complex $z$, which are known as {\em reflection formula} and {\em duplication formula} respectively. The proof of these properties can be found in \cite[Appendix A]{Ayoub}.

The well-known formula of Stirling for the Gamma function on a vertical strip states that for $a \leq \sigma\leq b$ and $t \geq 1$ (cf. \cite[p.~224]{cop}),
\begin{equation}\label{Stirling formula}
|\Gamma(\sigma+it)| = (2\pi)^{\frac{1}{2}} |t|^{\sigma-\frac{1}{2}} e^{-\frac{1}{2}\pi |t|}\left(1+ \mathcal{O}\left(\frac{1}{|t|} \right) \right).
\end{equation}

\subsection{Riemann zeta function}
The Riemann zeta function can be defined by the following series
\begin{equation}\label{zeta}
\zeta(s):= \sum_{n=1}^\infty \frac{1}{n^s},
\end{equation}
where $\re(s)>1$. It can be continued analytically to the whole complex plane except for the simple pole at $s=1$. The functional equation of the Riemann zeta function is given by \cite[p. 13, Equation (2.1.1)]{Titchmarsh}
\begin{equation}\label{zeta functional equation}
\zeta(s) = 2^s \pi^{s-1} \Gamma(1-s) \zeta(1-s) \sin\left(\frac{\pi s}{2} \right).
\end{equation}
Taking $k$-th exponent on both sides of \eqref{zeta}, we have for $\re(s)>1$,
\begin{equation}\label{zetak}
\sum_{n=1}^{\infty} \frac{d_k(n)}{n^s} = \zeta^k(s).
\end{equation}

\subsection{Special function}\label{sec:specialfunctions}
The mathematical functions which have more or less established names and notations due to their importance in mathematical analysis, functional analysis, geometry, physics, or other applications are known as special functions. These mainly appear as solutions of differential equations or integrals of elementary functions. 

One of the most important families  of special functions are the Bessel functions,  which are basically the canonical solution of Bessel's differential equations
\begin{equation*}
x^2\frac{d^2y}{dx^2}+x\frac{dy}{dx}+(x^2-a^2)y = 0
\end{equation*}
where $a$ is any arbitrary complex number.

Meijer in 1936 introduced a general special function namely, $G$-function (cf. \cite[p.~143]{Luke}), which includes most of the known special functions as particular cases. For any non-zero complex number $z$ and for integers $m, n, p, q$ satisfying $0\leq m <p$ and $0\leq n <q$, the Meijer $G$-function can be defined as an inverse Mellin transform of quotient of products of gamma factors as
\begin{equation}\label{Meijer G}
\begin{aligned}
G^{m, \ n}_{p, \ q}\bigg(\begin{matrix}
a_1, \ldots, a_p \\
b_1, \ldots, b_q
\end{matrix} \ \bigg|\ z\bigg)=\frac{1}{2\pi i}\underset{{(C)}}{\bigints} \frac{\prod\limits_{j=1}^m\Gamma(b_j-s)\prod\limits_{j=1}^n\Gamma(1-a_j+s)}{\prod\limits_{j=m+1}^q\Gamma(1-b_j+s)\prod\limits_{j=n+1}^p\Gamma(a_j-s)}z^s \rm{d}s,
\end{aligned}
\end{equation}
where $(C)$ in the integral denotes the vertical line from $C-i\infty$ to $C+i\infty$.

Special cases of the $G$-function include many other special functions. For instance, there are many formulae which yield relations between the $G$-function and the Bessel functions. Two important formulas among them, which we have used are given by \cite[p. 216, Equation (4)]{Bateman}, \cite[p. 675, Equation (13)]{Prudnikov}
\begin{align}
G_{0,2}^{2,0}\( \begin{array}{cl}
- \\
a, b
\end{array} \bigg | z \right) &= 2z^{\frac{1}{2}(a+b)} K_{a-b} (2z^{1/2}), \label{G and K function}\\
G_{0,4}^{3,0}\( \begin{array}{cl}
- \\
0, 0, \frac{1}{2}; \frac{1}{2}  
\end{array} \bigg | z \right) &= 4 \re\left(K_0\left(4 z^{\frac{1}{4}} e^{\frac{i\pi}{4}} \right) \right).\label{G and Ker}
\end{align}

\section{Generalization of Ramanujan's identity and it's special cases}\label{Generalization of Ramanujan's identity and it's special cases}
In this section, we prove Theorem \ref{Main result} and its special cases at $k=1$ and $2$. We begin with the following lemma.

\begin{lemma}\label{Lem:Inverse Mellin}
Let $\Psi_{\rho,k}(x)$ be defined as in \eqref{Psiomegakx}. Then for $c = \re(s) > 1$, we have
\begin{equation*}
\Psi_{\rho,k}(x) = \frac{1}{2\pi i} \int_{(c)}\Gamma^k(s) \zeta^k(s) \cos^{k-1}\left(\frac{
\pi s}{2}\right) (\rho x)^{-s} ds.
\end{equation*}
\end{lemma} 

\begin{proof}
The definition of Meijer $G$-function in \eqref{Meijer G} yields, for $\mu <0$,
\begin{align*}
G_{0,\ 2k}^{k+1,\ 0}\bigg(
\begin{matrix}
\text{---}\\
(0)_k,\frac{1}{2};\left(\frac{1}{2}\right)_{k-1}
\end{matrix} \ \bigg| \
\frac{\rho^2 j^2 x^2}{2^{2k}}
\bigg) &= \frac{1}{2\pi i} \int_{(\mu)} \frac{\Gamma^k(-s)\Gamma\left(\frac{1}{2}-s \right)}{\Gamma^{k-1}\left(\frac{1}{2}+s \right)} \left(\frac{\rho^2 j^2 x^2}{2^{2k}} \right)^s ds\\
&= \frac{1}{2\pi i} \int_{(-\mu)} \frac{\Gamma^k(s)\Gamma^k\left(\frac{1}{2}+s \right)}{\Gamma^{k-1}\left(\frac{1}{2}+s \right)\Gamma^{k-1}\left(\frac{1}{2}-s \right)} \left(\frac{\rho^2 j^2 x^2}{2^{2k}} \right)^{-s} ds,
\end{align*}
where in the last step we make the change of variable $s \mapsto -s$. Employing the reflection formula \eqref{Reflection formula} in the denominator and the duplication formula \eqref{Duplication formula} in the numerator of the above integral, we obtain
\begin{align}
G_{0,\ 2k}^{k+1,\ 0}\bigg(
\begin{matrix}
\text{---}\\
(0)_k,\frac{1}{2};\left(\frac{1}{2}\right)_{k-1}
\end{matrix} \ \bigg| \
\frac{\rho^2 j^2 x^2}{2^{2k}}
\bigg) &= \frac{2^k}{\pi^{\frac{k}{2}-1}}\frac{1}{2\pi i} \int_{(-\mu)} \Gamma^k(2s) \cos^{k-1}(\pi s) (\rho j x)^{-2s} ds \nonumber\\
&= \frac{2^{k-1}}{\pi^{\frac{k}{2}-1}}\frac{1}{2\pi i} \int_{(c)} \Gamma^k(s) \cos^{k-1}\left(\frac{\pi s}{2}\right) (\rho j x)^{-s} ds, \label{reduced Meijer G}
\end{align} 
where $c = -2\mu>0$. Invoking \eqref{reduced Meijer G} into the definition \eqref{Psiomegakx} of $\Psi_{\rho,k}(x)$, we obtain
\begin{align*}
\Psi_{\rho,k}(x) =\frac{1}{2\pi i} \sum_{j=1}^\infty d_k(j) \int_{(c)} \Gamma^k(s) \cos^{k-1}\left(\frac{\pi s}{2}\right) (\rho j x)^{-s} ds
\end{align*}
We next interchange the order of summation and integration in the above expression for $c= \re(s)>1$, which follows by applying \eqref{Reflection formula} on the cosine factor and then \eqref{Stirling formula} on Gamma factors. Therefore, for $c= \re(s)>1$, we have
\begin{align}
\Psi_{\rho,k}(x) &=\frac{1}{2\pi i} \int_{(c)} \Gamma^k(s) \left(\sum_{j=1}^\infty \frac{d_k(j)}{j^s}\right) \cos^{k-1}\left(\frac{\pi s}{2}\right) (\rho x)^{-s} ds \nonumber\\
&= \frac{1}{2\pi i} \int_{(c)}\Gamma^k(s) \zeta^k(s) \cos^{k-1}\left(\frac{
\pi s}{2}\right) (\rho x)^{-s} ds, \label{reduced Psiomegakx}
\end{align}
where in the last step we have applied \eqref{zetak}. This completes the proof of the lemma.
\end{proof}

\subsection{Proof of Theorem \ref{Main result}}
We begin with the following infinite series :
\begin{equation}\label{Lmomega}
\mathscr{L}_m(\rho) := \sum_{n=1}^{\infty}\frac{d_k(n)}{n^{2m+1}}\Psi_{\rho, k}(n),
\end{equation}
where $d_k(n)$ is the Piltz divisor function. Applying Lemma \ref{Lem:Inverse Mellin}, we have for $1<c<3$,
\begin{align*}
\mathscr{L}_m(\rho) &= \frac{1}{2\pi i} \sum_{n=1}^{\infty}\frac{d_k(n)}{n^{2m+1}} \int_{(c)}\Gamma^k(s) \zeta^k(s) \cos^{k-1}\left(\frac{\pi s}{2}\right) (\rho n)^{-s} \ ds\\
&= \frac{1}{2\pi i} \int_{(c)} \Gamma^k(s) \zeta^k(s) \left( \sum_{n=1}^{\infty}\frac{d_k(n)}{n^{2m+1+s}} \right) \cos^{k-1}\left(\frac{\pi s}{2}\right) \rho^{-s} \ ds\\
&= \frac{1}{2 \pi i}\int_{(c)}\Gamma^{k}(s) \zeta^k(s) \zeta^{k}(2m+1+s) \cos^{k-1}\left(\frac{\pi s}{2}\right) \rho^{-s} \ ds,
\end{align*}
where in the penultimate step, the interchange of the order of summation and integration is justified similarly as was done in \eqref{reduced Psiomegakx} and in the last step we have applied \eqref{zetak}. Employing functional equation \eqref{zeta functional equation} of Riemann zeta function into the above integral, we obtain
\begin{align}\label{Reduced lmomega}
\mathscr{L}_m(\rho) = \frac{1}{2^k}\frac{1}{2\pi i}\int_{(c)} F_m(s) \ ds,
\end{align}
where
\begin{equation*}
F_m(s)= \frac{\zeta^k(2m+s+1)\zeta^k(1-s)}{\cos\left(\frac{\pi s}{2}\right)} \left(\frac{\rho}{(2\pi)^k}\right)^{-s}.
\end{equation*}
We next consider the contour $\mathcal{C}$ determined by the line segments $[c-iT,c +iT], [c+iT, \lambda + iT], [\lambda+iT, \lambda-iT], [\lambda -iT, c-iT]$, where $-2m-3<\lambda<-2m-1$. It can be easily seen that the integrand $F_m(s)$ in the above integral has poles of order $k$ at $s=0$ and at $s=-2m$ due to the poles of each of the zeta functions. The integrand $F_m(s)$ also has simple poles at 
$1-2j$ for $0\leq j \leq m+1$ due to the zeros of cosine factor in the denominator. Note that the zeros of the cosine factor at $1-2j$ for $j >m+1$ do not contribute any poles since it get cancelled with the trivial zeros of $\zeta(2m+1+s)$. 

Letting $R_a$ be the corresponding residue of $F_m(s)$ at $s=a$, the Cauchy residue theorem yields
\begin{align}
\frac{1}{2\pi i}\int_{\mathcal{C}} F_m(s) \ ds &= \frac{1}{2\pi i} \left[\int_{c-iT}^{c+iT}+\int_{c+iT}^{\lambda+iT}+\int_{\lambda+iT}^{\lambda-iT}+\int_{\lambda-iT}^{c-iT} \right] F_m(s) \ ds \nonumber\\
&= R_0 + R_{-2m}+ \sum_{j=0}^{m+1} R_{1-2j} \label{Cauchy residue theorem}
\end{align}
We next evaluate the residues on the right hand side of \eqref{Cauchy residue theorem}. The residue at $s=0$ is given by
\begin{align}
R_{0}&=\frac{1}{(k-1)!}\lim_{s\rightarrow 0}\frac{d^{k-1}}{ds^{k-1}}\left(\frac{s^k \zeta^k(2m+1+s)\zeta^k(1-s)}{\cos\left(\frac{\pi s}{2}\right)}\left(\frac{\rho}{(2\pi)^k}\right)^{-s}\right)\nonumber\\
&= \frac{2^k}{(k-1)!}\lim_{s\rightarrow 0}\frac{d^{k-1}}{ds^{k-1}}\left( s^k \zeta^k(2m+1+s)\zeta^k(s)\Gamma^k(s)\cos^{k-1}\left(\frac{\pi s}{2}\right)\rho^{-s}\right)\nonumber\\
&= \frac{2^k}{(k-1)!}\frac{d^{k-1}}{ds^{k-1}}\left( \zeta^k(2m+1+s)\zeta^k(s)\Gamma^k(s+1)\cos^{k-1}\left(\frac{\pi s}{2}\right)\rho^{-s}\right)\bigg\rvert_{s=0},\label{R0}
\end{align}
where in the penultimate step we used the functional equation \eqref{zeta functional equation} of the Riemann zeta function. The residue at $s=-2m$ evaluates as 
\begin{align}
R_{-2m}&= \frac{1}{(k-1)!}\lim_{s\rightarrow -2m}\frac{d^{k-1}}{ds^{k-1}}\left(\frac{(s+2m)^k \zeta^k(2m+1+s)\zeta^k(1-s)}{\cos\left(\frac{\pi s}{2}\right)}.\left(\frac{\rho}{(2\pi)^k}\right)^{-s}\right)\nonumber\\
&= \frac{(-1)^{m+1}}{(k-1)!} \left(\frac{\rho}{(2\pi)^k} \right)^{2m}\lim_{s\rightarrow 0}\frac{d^{k-1}}{ds^{k-1}}\left(\frac{s^k \zeta^k(2m+1+s)\zeta^k(1-s)}{\cos\left(\frac{\pi s}{2}\right)}\left(\frac{\rho}{(2\pi)^k}\right)^{s}\right)\nonumber\\
&= \frac{(-1)^{m+1}2^k}{(k-1)!} \left(\frac{\rho}{(2\pi)^k} \right)^{2m}\hspace{-3pt}\frac{d^{k-1}}{ds^{k-1}}\left(\zeta^k(2m+1+s)\zeta^k(s)\Gamma^k(s+1)\cos^{k-1}\left(\frac{\pi s}{2}\right)\left(\frac{\rho}{(4\pi^2)^k}\right)^{s}\right)\bigg\rvert_{s=0},\label{R-2m}
\end{align}
where in the penultimate step we made the change of variable $s \mapsto -s-2m$ and in the last step we have used \eqref{zeta functional equation}. The residue at $s=1-2j$ with $0\leq j \leq m+1$, can be calculated similarly as the previous residue calculation, which is given by
\begin{align}
R_{1-2j} &= \lim_{s\rightarrow 1-2j}\left(\frac{(s-1+2j)\zeta^{k}(2m+1+s)\zeta^k(1-s)}{\cos\left(\frac{\pi s}{2}\right)}.\left(\frac{\rho}{(2\pi)^k}\right)^{-s}\right)\nonumber\\
&= (-1)^{j+1}\frac{2}{\pi}\zeta^{k}(2m+2-2j)\zeta^k(2j)\left(\frac{\rho}{(2\pi)^k}\right)^{(2j-1)}\nonumber\\
&=(-1)^{km+k+j+1} 2^{2km+1}\pi^{2km+2k-1}\frac{B^k_{2m-2j+2}B^k_{2j}}{(2m-2j+2)!^k(2j)!^k}\left(\frac{\rho}{(2\pi)^k}\right)^{(2j-1)},\label{R1-2j}
\end{align}
where in the last step we applied \eqref{Euler's formula} on the zeta factors. 

It can be easily shown that the horizontal integrals in \eqref{Cauchy residue theorem} vanish as $T\to \infty$ by applying  \eqref{Stirling formula} on the Gamma factors and an elementary bound on zeta factors in $F_m(s)$. Therefore, letting $T\to \infty$, \eqref{Reduced lmomega} and  \eqref{Cauchy residue theorem} together imply
\begin{align}
\mathscr{L}_m(\rho)=\frac{1}{2^k}\left[R_0+R_{-2m}+\sum_{j=0}^{m+1}R_{1-2j}+\frac{1}{2\pi i}\int_{(\lambda)}F_m(s) \ ds\right]. \label{Simplied Cauchy residue}
\end{align}
We next handle the vertical integral on the right-hand side of the above equation. Substituting $s$ by $-2m-s$, the integral reduces to
\begin{align}
\frac{1}{2\pi i}\int_{(\lambda)}F_m(s) \ ds &= (-1)^m
\left(\frac{\rho}{(2\pi)^k}\right)^{2m}\frac{1}{2\pi i}\int_{(c^{'})}\frac{\zeta^k(1-s)\zeta^k(2m+1+s)}{\cos\left(\frac{\pi s}{2}\right)}\left(\frac{\rho}{(2\pi)^k}\right)^{s} \ ds \nonumber\\
&= (-1)^m \left(\frac{\rho}{(2\pi)^k}\right)^{2m}\frac{1}{2\pi i}\int_{(c^{'})}\frac{\zeta^k(1-s)\zeta^k(2m+1+s)}{\cos\left(\frac{\pi s}{2}\right)}\left(\frac{\frac{(4\pi^2)^k}{\rho}}{(2\pi)^k}\right)^{-s} \ ds \nonumber\\
&= (-1)^m 2^k \left(\frac{\rho}{(2\pi)^k}\right)^{2m}\mathscr{L}_m\left(\frac{(4\pi^2)^k}{\rho}\right), \label{Vertical integral}
\end{align}
where $c' = -2m-\lambda$ satisfying $1<c'<3$. Inserting \eqref{R0}, \eqref{R-2m}, \eqref{R1-2j} and \eqref{Vertical integral} together into \eqref{Simplied Cauchy residue}, we obtain
\begin{align*}
\mathscr{L}_m(\rho)&-\frac{1}{(k-1)!}\frac{d^{k-1}}{ds^{k-1}}\left( \zeta^k(2m+1+s)\zeta^k(s)\Gamma^k(s+1)\cos^{k-1}\left(\frac{\pi s}{2}\right)\rho^{-s}\right)\bigg\rvert_{s=0} =(-1)^m \left(\frac{\rho}{(2\pi)^k}\right)^{2m}\\
&\times \left[\mathscr{L}_m\left(\frac{(4\pi^2)^k}{\rho}\right) -\frac{1} {(k-1)!} \frac{d^{k-1}}{ds^{k-1}}\left(\zeta^{k}(2m+1+s)\zeta^k(s)\Gamma^{k}(s+1)\cos^{k-1}\left(\frac{\pi s}{2}\right)\left(\frac{\rho}{(4\pi^2)^k}\right)^{s}\right)\bigg\rvert_{s=0}\right]\\
&+(-1)^{km+k+1}2^{2km-k+1}\pi^{2km+2k-1}\sum_{j=0}^{m+1}(-1)^{j}\frac{B^k_{2m-2j+2}B^k_{2j}}{(2m-2j+2)!^k(2j)!^k}\left(\frac{\rho}{(2\pi)^k}\right)^{2j-1}
\end{align*}
We next make the change of variable $j \mapsto m+1-j$ into the finite sum on the right-hand side of the above equation to reduce the above equation as
\begin{align*}
\mathscr{L}_m(\rho)&-\frac{1}{(k-1)!}\frac{d^{k-1}}{ds^{k-1}}\left( \zeta^k(2m+1+s)\zeta^k(s)\Gamma^k(s+1)\cos^{k-1}\left(\frac{\pi s}{2}\right)\rho^{-s}\right)\bigg\rvert_{s=0} =(-1)^m \left(\frac{\rho}{(2\pi)^k}\right)^{2m}\\
&\times \left[\mathscr{L}_m\left(\frac{(4\pi^2)^k}{\rho}\right) -\frac{1} {(k-1)!} \frac{d^{k-1}}{ds^{k-1}}\left(\zeta^{k}(2m+1+s)\zeta^k(s)\Gamma^{k}(s+1)\cos^{k-1}\left(\frac{\pi s}{2}\right)\left(\frac{\rho}{(4\pi^2)^k}\right)^{s}\right)\bigg\rvert_{s=0}\right]\\
&+(-1)^{km+k+m}2^{1-2k}\pi^{k-1}\sum_{j=0}^{m+1}(-1)^{j}\frac{B^k_{2m-2j+2}B^k_{2j}}{(2m-2j+2)!^k(2j)!^k} \ (2\pi)^{2kj} \ \rho^{2m-2j+1}
\end{align*}
Letting $\alpha^k=\frac{\rho}{2^k} $, $\beta^k=\frac{2^k \pi^{2k}}{\rho}$ and multiplying both sides by $\alpha^{-km}$ in the above equation, we arrive at
\begin{align*}
&(\alpha^k)^{-m}\left[\sum_{n=1}^{\infty}\frac{d_k(n)}{n^{2m+1}}\Psi_{(2\alpha)^k, k}(n)-\frac{1}{(k-1)!}\frac{d^{k-1}}{ds^{k-1}}\left( \zeta^k(2m+1+s)\zeta^k(s)\Gamma^k(s+1)\cos^{k-1}\left(\frac{\pi s}{2}\right)(2\alpha)^{-ks}\right)\bigg\rvert_{s=0}\right]\\
&=(-\beta^k)^{-m}\left[
\sum_{n=1}^{\infty}\frac{d_k(n)}{n^{2m+1}}\Psi_{(2\beta)^k, k}(n)-\frac{1} {(k-1)!}\frac{d^{k-1}}{ds^{k-1}}\left(\zeta^k(2m+1+s)\zeta^{k}(s)\Gamma^{k}(s+1)\cos^{k-1}\left(\frac{\pi s}{2}\right)(2\beta )^{-ks}\right)\bigg\rvert_{s=0}\right]\\
&+(-1)^{km+k+m}\left(\frac{\pi}{2}\right)^{k-1} 2^{2km}\sum_{j=0}^{m+1}(-1)^{j}\frac{B^k_{2m-2j+2}B^k_{2j}}{(2m-2j+2)!^k(2j)!^k} \alpha^{k(m+1-j)}\beta^{kj},
\end{align*}
where we applied the definition \eqref{Lmomega} of $\mathscr{L}_m(\rho)$. This completes the proof of our theorem. \qed

We next show that our formula provides Ramanujan's identity at $k=1$.
\subsection{Proof of Corollary \ref{Ramanujan theorem}}
It follows from the definition \eqref{Psiomegakx} that at $k=1$, we have
\begin{align*}
\Psi_{\rho, 1}(x) &=\frac{1}{\sqrt{\pi}} \sum_{j=1}^{\infty} 
G_{0,\ 2}^{2,\ 0}\bigg(
\begin{matrix}
\text{---}\\
0,\frac{1}{2}
\end{matrix} \ \bigg| \
\frac{\rho^2 j^2 x^2}{2^{2}}
\bigg)\\
&= \sqrt{\frac{2}{\pi}} \sum_{j=1}^{\infty}  (\rho j x)^{1/2}K_{1/2}(\rho j x),
\end{align*}
where in the last step we used \eqref{G and K function}. Now applying \cite[p.~254, Equation (10.39.2)]{NIST}, we obtain
\begin{equation*}
\Psi_{\rho, 1}(x) = \sum_{j=1}^{\infty} e^{-\rho j x}
\end{equation*}
Therefore invoking the above relation into \eqref{General Ramanujan identity}, we conclude our corollary.
\qed

We next prove that the identity \eqref{Dixit formula} can be obtained as a special case of Theorem \ref{Main result} for $k=2$.
\subsection{Proof of Corollary \ref{Dixit theorem}}
Employing \eqref{G and Ker} in the definition \eqref{Psiomegakx} at $k=2$, we obtain
\begin{align*}
\Psi_{\rho,2}(x)&=\frac{1}{2}\sum_{j=1}^{\infty}d(j)
G_{0,4}^{3,0}\left(
\begin{array}{c}
---\\
0, 0,\frac{1}{2};\frac{1}{2}
\end{array}\middle\vert
\frac{\rho^2j^2x^2}{16}
\right)\\
&=2 \sum_{j=1}^{\infty}d(j)\ \re\left(K_0\left(2\epsilon\sqrt{\rho jx}\right)\right)\\
&= \sum_{j=1}^{\infty}d(j) \left(K_0\left(2\epsilon\sqrt{\rho jx}\right) + K_0\left(2\overline{\epsilon}\sqrt{\rho jx}\right)\right),
\end{align*}
where $\epsilon = e^{\frac{i\pi}{4}}$. We next insert the above relation into \eqref{General Ramanujan identity} and make the change of variable $j \mapsto m+1-j$ in the finite sum on the right-hand side of \eqref{General Ramanujan identity} to conclude our corollary.

\section{Transformation of the generalized Dedekind eta function}\label{Transformation of the generalized Dedekind eta function}
In this section we mainly prove Theorem \ref{Generalization of Dedekind eta}.
\subsection{Proof of Theorem \ref{Generalization of Dedekind eta}}
Let 
\begin{equation}\label{Lomega}
\mathscr{L}(\rho) := \sum_{n=1}^{\infty}\frac{d_k(n)}{n}\Psi_{\rho, k}(n),
\end{equation}
where $d_k(n)$ is the Piltz divisor function. We next handle the series $\mathscr{L}(\rho)$ similarly as $\mathscr{L}_m(\rho)$, which was derived in the proof of Theorem \ref{Main result}. Proceeding similarly as in \eqref{Reduced lmomega}, the series $\mathscr{L}(\rho)$ transforms into
\begin{align*}
\mathscr{L}(\rho) = \frac{1}{2^k}\frac{1}{2\pi i}\int_{(c)} \frac{\zeta^k(1+s)\zeta^k(1-s)}{\cos\left(\frac{\pi s}{2}\right)} \left(\frac{\rho}{(2\pi)^k}\right)^{-s} \ ds,
\end{align*}
where $1<c<3$. We next shift the line of integration to $-3<\lambda<-1$ by constructing a rectangular contour. Thus the integrand has a pole of order $2k$ at $s=0$ due to the poles of zeta functions in the numerator as well as simple poles at $s=-1$ and $s=1$ due to the zeros of the cosine factor in the denominator.

Letting $R_a$ be the corresponding residue at $s=a$, the residues at the above poles can be evaluated as
\begin{align*}
R_0=(-1)^k\frac{2^{2k}}{(2k-1)!}&\frac{d^{2k-1}}{ds^{2k-1}} \left( \Gamma^k(1+s)\Gamma^k(1-s)\zeta^k(s)\zeta^k(-s)\cos^{2k-1}\left(\frac{\pi s}{2}\right)\left(\frac{\rho}{(2\pi)^k}\right)^{-s}\right)\bigg\rvert_{s=0},\\
R_{1}&=2\frac{(-1)^{k+1}  \pi^{3k-1}}{6^k \rho} \hspace{.5cm} \text{and} \hspace{.5cm}
R_{-1}=2\frac{(-1)^k \pi^{k-1}\rho}{(24)^k}
\end{align*}

Thus applying Cauchy residue theorem and simplifying further as in the proof of Theorem \ref{Main result}, we arrive at
\begin{align*}
\mathscr{L}(\rho) - \mathscr{L}\left(\frac{(4\pi^{2})^k}{\rho}\right)&= \frac{(-1)^k 2^{k}}{(2k-1)!}\frac{d^{2k-1}}{ds^{2k-1}} \left( \Gamma^k(1+s)\Gamma^k(1-s)\zeta^k(s)\zeta^k(-s)\cos^{2k-1}\left(\frac{\pi s}{2}\right)\left(\frac{\rho}{(2\pi)^k}\right)^{-s}\right)\bigg\rvert_{s=0}\\
&+ 2\frac{(-1)^{k+1}\pi^{k-1}}{24^k}\left(\frac{2^k\pi^{2k}}{\rho} - \frac{\rho}{2^k}\right).
\end{align*}
Finally letting $\alpha^k=\frac{\rho}{2^k} $, $\beta^k=\frac{2^k \pi^{2k}}{\rho}$ and applying the definition of $\mathscr{L}(\rho)$ in \eqref{Lomega}, we can conclude our theorem.

\subsection*{Acknowledgements} The authors show their sincere gratitude to Prof. Atul Dixit for fruitful suggestions and discussions. The first author is a National Postdoctoral Fellow (NPDF) at IIT Gandhinagar funded by the grant PDF/2021/001224 and the second author is a M.Sc student at IIT Gandhinagar. Both the authors would like to thank IIT Gandhinagar for the support during this project.


\begin{thebibliography}{99}
\bibitem{Apery1} R. Ap{\'e}ry, {\em Irrationalit{\'e} de $\zeta(2)$ et $\zeta(3)$}, Ast{\'e}risque {\bf 61} (1979), 11--13.

\bibitem{Apery2}  R. Ap{\'e}ry, {\em Interpolation de fractions continues et irrationalit{\'e} de certaines constantes}, in: Bull.
Section des Sci., Tome III, Biblioth{\'e}que Nationale, Paris, 1981, 37--63.

\bibitem{Ayoub} R. G. Ayoub, {\em An introduction to the analytic theory of numbers}, American Math. Soc., (1963).

\bibitem{Bateman} H. Bateman, A. Erd\'elyi, {\em Higher Transcendental Functions}, {\bf Vol. I}, New York: McGraw-Hill, (1953).

\bibitem{Berdnt0} B.C. Berndt, {\em Modular transformations and generalizations of several formulae of Ramanujan}, Rocky
Mountain J. Math. {\bf 7} (1977) 147--189.

\bibitem{Berdnt}  B.C. Berndt, A. Straub, {\em Ramanujan’s formula for $\zeta(2n + 1)$}, in: H. Montgomery, A. Nikeghbali, M. Rassias (Eds.), Exploring the Riemann Zeta Function, Springer, Cham (2017), 13--34.

\bibitem{cop} E.~T.~Copson, \emph{Theory of Functions of a Complex Variable}, Oxford University Press, Oxford, 1935.

\bibitem{Dixit} A. Dixit and R. Gupta, {\em On squares of odd zeta values and analogues of Eisenstein series}, Advances in Applied Mathematics {\bf 110} (2019) 86--119.

\bibitem{Ferrar} W.L. Ferrar, {\em Some solutions of the equation $F(t) = F(t^{-1})$}, J. Lond. Math. Soc. {\bf 11} (1936) 99--103.

\bibitem{Grosswald1} E. Grosswald, {\em Die Werte der Riemannschen Zetafunktion an ungeraden Argumentstellen}, Nachr. Akad. Wiss. Göttingen Math.-Phys. Kl. II (1970) 9--13.

\bibitem{Grosswald2} E. Grosswald, {\em Comments on some formulae of Ramanujan}, Acta Arith. {\bf 21} (1972) 25--34.

\bibitem{Kirschenhofer} P. Kirschenhofer, H. Prodinger, {\em On some applications of formulae of Ramanujan in the analysis of
algorithms}, Mathematika {\bf 38 (1)} (1991) 14--33.

\bibitem{Koshliakov} N.S. Koshliakov, {\em On an extension of some formulae of Ramanujan}, Proc. Lond. Math. Soc. {\bf 2 (1)}
(1936) 26--32.

\bibitem{Lerch} M. Lerch, {\em Sur la fonction $\zeta(s)$ pour valeurs impaires de l'argument}, J. Sci. Math. Astron. pub. pelo
Dr. F. Gomes Teixeira, Coimbra {\bf 14} (1901) 65--69.

\bibitem{Malurkar} S.L. Malurkar, {\em On the application of Herr Mellin’s integrals to some series}, J. Indian Math. Soc. 16
(1925/1926) 130--138.

\bibitem{Oberhettinger} F. Oberhettinger, K. Soni, On some relations which are equivalent to functional equations involving
the Riemann zeta function, Math. Z. {\bf 127} (1972) 17--34.

\bibitem{Luke} Y. Luke, {\em The Special Functions and Their Approximations}, Vol. 1, UK Edition, Academic Press, INC. 1969.

\bibitem{NIST} F. W. J. Olver, D. W. Lozier, R. F. Boisvert and C. W. Clark, eds., {\em NIST Handbook of Mathematical
Functions}, Cambridge University Press, Cambridge, 2010.

\bibitem{Prudnikov} A. P. Prudnikov, Yu. A. Brychkov and O. I. Marichev, {\em Integrals and series}, Vol. 3: More Special Functions, Gordon and Breach, New York, 1990.

\bibitem{Ramanujan} S. Ramanujan, Notebooks (2 Volumes), Tata Institute of Fundamental Research, Bombay, 1957,
second ed., 2012.

\bibitem{Sullivan} C. O'Sullivan, {\em Formulas for non-holomorphic Eisenstein series and for the Riemann zeta function at odd integers}, Res. Number Theory {\bf 4 (3)} (2018) 36.

\bibitem{Titchmarsh} E.~Titchmarsh, The Theory of the Riemann Zeta Function, Clarendon Press, Oxford, 1986.

\bibitem{Voronoi} G.F. Vorono{\"i}, Sur une fonction transcendante et ses applications {\`a} la sommation de quelques s{\'e}ries, Ann. {\'E}c. Norm. Sup{\'e}r. (3) {\bf 21} (1904) 207--267, 459--533.

\bibitem{Watson} G.N. Watson, {\em A Treatise on the Theory of Bessel Functions}, second ed., Cambridge University Press, London, 1944.

\bibitem{Zudilin} V. Zudilin, {\em One of the numbers $\zeta(5)$, $\zeta(7)$, $\zeta(9)$, $\zeta(11)$ is irrational} (Russian) Uspekhi Mat. Nauk {\bf 56} (2001), no. 4 (340), 149--150; translation in Russian Math. Surveys {\bf 56} (2001), no. 4, 774--776.
\end{thebibliography}
\end{document}